\documentclass[onecolumn]{IEEEtran}
\usepackage{amsmath,amsfonts,amssymb} 
\usepackage[utf8]{inputenc}
\usepackage{bbm}
\usepackage{amsmath}
\usepackage{amssymb}
\usepackage{titlesec}
\usepackage{array}
\usepackage{textcomp}
\usepackage{stfloats}
\usepackage{url}
\usepackage{verbatim}
\usepackage{graphicx}
\usepackage{cite}
\usepackage{amsthm}%
\usepackage{mathrsfs}%
\usepackage{xcolor}%
\usepackage{textcomp}%
\usepackage{manyfoot}
\usepackage{booktabs}%
\usepackage{syntax}
\usepackage{fancybox}
\usepackage{adjustbox}
\usepackage{listings}
\usepackage{multirow}
\usepackage{multicol}
\usepackage{comment}
\usepackage{hyperref}
\usepackage{nameref}
\usepackage{mathtools}
\usepackage{tikz-cd}
\usepackage{float}
\usepackage{titlesec}
\titleformat{\subparagraph}[runin]{\normalfont\normalsize\bfseries}{\thesubparagraph}{1em}{}
\usepackage{tocloft}
\newtheorem{theorem}{Theorem}%
\newtheorem{remark}{Remark}%

\begin{document}

\title{Compact Quantum Group Extensions of $USp_q(2n)$, $O_q(n)$, and $SO_q(2n)$}
\author{ Manabendra Giri\\Indian Institute of Technology Bombay, Powai, Mumbai\\ \texttt{manab@math.iitb.ac.in}
}

\IEEEpubid{0000--0000~\copyright~2023 IEEE}

\maketitle
\begin{abstract}
    I introduce compact quantum group extensions associated with the $q$-deformations of the classical compact groups $USp(2n)$, $O(n,\mathbb{R})$, and $SO(2n,\mathbb{R})$. Motivated by the relationship between $SU_q(n)$ and $U_q(n)$, I study the problem of constructing compact quantum groups $Z_{q,n}$ extending the standard compact quantum groups $A_{q,n}\in\{ {USp_q(2n), O_q(N), SO_q(2n)}\}$ through an additional central unitary element.
\end{abstract}


\tableofcontents

\section{Introduction}\label{sec1}
In theory of compact quantum group (CQG) developed by Woronowicz, the most accessible and widely studied examples in the literature are $SU_q(n)$ for $n>1,~q > 0$ and $q\neq 1$. Given the compact quantum group $SU_q(n)$, one has the compact quantum group $U_q(n)$, along with quantum group homomorphisms \begin{align*}
    \varrho_{q,n}:C(U_q(n))\to C(SU_q(n))& \text{ and }\vartheta_{q,n}:C(SU_q(n+1))\to C(U_q(n)),
\end{align*} such that $SU_q(n)$ is the quantum subgroup of $U_q(n)$, and $U_q(n)$ is quantum subgroup of $SU_q(n+1)$. Let $\varphi_{q,n}=\varrho_{q,n}\circ \vartheta_{q,n}$ and $\widetilde{\varphi}_{q,n}=\vartheta_{q,n}\circ\varrho_{q,n+1}$. I can therefore rephrase the situation as follows:
Given  compact quantum groups $SU_q(n)$ along with quantum subgroup maps $\varphi_{q,n}$, there exist compact quantum groups $U_q(n)$ along with quantum subgroup maps $\varrho_{q,n}$ and $\widetilde{\varphi}_{q,n}$, such that $\varphi_{q,n}\circ\varrho_{q,n+1}=\varrho_{q,n}\circ\widetilde{\varphi}_{q,n}$. I also have the compact quantum groups $USp_q(2n),~O_q(2n),~O_q(2n+1),~SO_q(2n)$ and $SO_q(2n+1)$ (denoted collectively as $A_{q,n}$) together with the quantum subgroup maps $\varphi_{q,n}:C(A_{q,n+1})\to C(A_{q,n}) $. One may ask whether there exist compact quantum groups $Z_{q,n}$, along with quantum subgroup maps \begin{align*}\varrho_{q,n}:C(Z_{q,n})\to C(A_{q,n})&\text{ and }\widetilde{\varphi}_{q,n}:C(Z_{q,n+1})\to C(Z_{q,n}),\end{align*} such that the following compatibility condition holds: \begin{align*}\varphi_{q,n}\circ\varrho_{q,n+1}=\varrho_{q,n}\circ\widetilde{\varphi}_{q,n}.\end{align*}

Actually, I want $C(Z_{q,n})$ to be such that the following relations in $C(A_{q,n})$:
\begin{align*}
    \sum_k V_{i,k}S_{k,j}=\delta_{i,j}
    \qquad \text{and} \qquad
    \sum_k S_{i,k}V_{k,j}=\delta_{i,j}
\end{align*}
are replaced by
\begin{align*}
    \sum_k V_{i,k}S_{k,j}=\delta_{i,j}\mathscr{Q}
    \qquad \text{and} \qquad
    \sum_k S_{i,k}V_{k,j}=\delta_{i,j}\mathscr{Q},
\end{align*}
for some unitary element $\mathscr{Q}$ satisfying
\[
\varrho_{q,n}(\mathscr{Q})=1.
\]

I first describe classical analogs $\widetilde{USp}(2n)$, $\widetilde{O}(N,\mathbb{R})$, and $\widetilde{SO}(2n,\mathbb{R})$, obtained by relaxing the defining relations up to multiplication by a unit scalar. 
I then construct their $q$-deformations using the Faddeev--Reshetikhin--Takhtajan framework associated with suitable $R$-matrices of types  $C$ and $D$. 
For each case, I define the corresponding bialgebras, look at the central group-like element $\mathscr{Q}_q$, and construct the antipode and involution explicitly.
This yields Hopf-$*$-algebra structures whose $C^*$-completions define compact matrix quantum groups.

In particular, I obtain the compact quantum groups $\widetilde{USp}_q(2n)$, $\widetilde{O}_q(N)$, and $\widetilde{SO}_q(2n)$ together with natural quantum subgroup maps onto $USp_q(2n)$, $O_q(N)$, and $SO_q(2n)$, respectively. I also prove a general structural theorem for universal $C^*$-algebra extensions generated by a central unitary element. As a consequence, every irreducible representation of the extended quantum groups is obtained from an irreducible representation of the corresponding standard compact quantum group by twisting with a one-dimensional unitary element. This provides a complete description of the irreducible representation theory of these extensions in terms of the known representation theory of the underlying compact quantum groups.
\section{Groups  $\widetilde{USp}(2n)$, $\widetilde{O}(2n+1)$, $\widetilde{O}(2n)$ and $\widetilde{SO}(2n)$}
 I can define a map $\Xi_n:U(n)\longrightarrow SU(n+1)$ by $A\overset{\Xi_n}{\longmapsto}\begin{bmatrix}
     A& 0\\0&det(A)^{-1}\end{bmatrix}$. Then $\Xi_n(U(n))$ is a compact group. Therefore one can use the term $C(U(n))$ to mean all continuous functions on $\Xi_n(U(n))$. Therefore $C(U(n))$ is a compact quantum group. 

Here first I want to define the matrix group $\widetilde{USp}(2n)$, $\widetilde{O}(2n+1)$, $\widetilde{O}(2n)$ and $\widetilde{SO}(2n)$.

Consider two maps $\Upsilon_{n}:GL(2n,\mathbb{C})\longrightarrow GL(2n,\mathbb{C})$ and $\wp_n:GL(n,\mathbb{C})\longrightarrow GL(n,\mathbb{C})$ such that
$\Upsilon_{n}(M)=S_{2n}MS_{2n}^{-1}$. and $\wp_n(M)=(\sqrt{D}_{n,-}) Q_{n}^{tr}MQ_{n}(\sqrt{D}_{n,+})$ where $S_{2n},~Q_{n},~\sqrt{D}_{n,+}$ and $\sqrt{D}_{n,-}$ are given below.

Let $diag(a_{1,1},a_{2,2},\cdots,a_{n,n})$ denote the diagonal matrix of order $n$ and $cdiag(a_{1,n},a_{2,n-1},\cdots,a_{n-1,2},a_{n,1})$ denote the cross diagonal matrix of order $n$. Let $C_n=cdiag(1,1,\cdots,1)$ be a cross diagonal matrix of order $n$. 
Consider the matrices

\begin{tabular}{ccc}
   $J_{2n}=\begin{bmatrix}O&I_n\\-I_n&0\end{bmatrix}$ ,  &$K_{2n}=\begin{bmatrix}O&C_n\\-C_n&0\end{bmatrix}$,& $S_{2n}=\begin{bmatrix}I_n&0\\0&C_n\end{bmatrix}$,\\ \\
    
     $Q_{2n}=\frac{1}{\sqrt{2}}\begin{bmatrix}I_n&I_n\\C_n&-C_n\end{bmatrix}$, & $\sqrt{D}_{2n,+}=\begin{bmatrix}I_n&0\\0&iI_n\end{bmatrix}$, & $\sqrt{D}_{2n,-}=\begin{bmatrix}I_n&0\\0&-iI_n\end{bmatrix}$, \\\\
     
     $Q_{2n+1}=\frac{1}{\sqrt{2}}\begin{bmatrix}I_n&0&I_n\\0&1&0\\C_n&0&-C_n\end{bmatrix}$, &$\sqrt{D}_{2n+1,+}=\begin{bmatrix}I_{n+1}&0\\0&iI_n\end{bmatrix}$, & $\sqrt{D}_{2n+1,-}=\begin{bmatrix}I_{n+1}&0\\0&-iI_n\end{bmatrix}$.
\end{tabular}

I know that compact symplectic group $USp(2n)$ is $\{M\in U(2n)~:~ MJ_{2n} M^{t}=J_{2n}\}$.  Here  I use the term $C(USp(2n))$ to mean all continuous functions on $\Xi_{2n}\circ\Upsilon_{n}(USp(2n))$.

  Here  I use the term $C(O(n,\mathbb{R}))$ ( $C(SO(n,\mathbb{R}))$  ) to mean all continuous functions on $\Xi_n\circ\wp_n(O(n,\mathbb{R}))$ ($\Xi_n\circ\wp_n(SO(n,\mathbb{R}))$ respectively ).

 Therefore I have the followings
\begin{align}
    C(USp(2n))&=\text{ continuous functions on }\Xi_{2n}(\{M\in~U(2n)~:~MK_{2n} M^{t}K_{2n}^t=I_{2n}\})\\
    C(O(n,\mathbb{R}))&=\text{ continuous functions on }\Xi_{n}(\{M\in~U(n)~:~MC_{n} M^{t}C_{n}^t=I_{n}\})\\
    C(SO(n,\mathbb{R}))&=\text{ continuous functions on }\Xi_{n}(\{M\in~U(2n)~:~MC_{n} M^{t}C_{n}^t=I_{n},~det(M)=1\})
\end{align}

Therefore I define the groups
\begin{align}
    \widetilde{USp}(2n)&=\{M\in U(2n)~:~MJ_{2n}M^t=\lambda J_{2n}~\&~|\lambda |=1 \}\\&=\Upsilon_n^{-1}(\{M\in~U(2n)~:~MK_{2n} M^{t}K_{2n}^t=\lambda I_{2n}~\&~|\lambda |=1\})\nonumber\\~~\nonumber\\
    \widetilde{O}(n,\mathbb{R})&=\wp_n^{-1}(\{M\in~U(n)~:~MC_{n} M^{t}C_{n}^t=\lambda I_{n}~\&~|\lambda |=1\})
\end{align}
Since for any $M\in \widetilde{O}(2n,\mathbb{R})$ I have $(\det M)^2=\lambda^{2n}$, there are two possible branches:
\[
\det(M)=\pm \lambda^n.
\]
I define $\widetilde{SO}(2n,\mathbb{R})$ using the positive branch, as the natural analogue of the special orthogonal group, i.e.,
\begin{align}
    \widetilde{SO}(2n,\mathbb{R})&=\wp_n^{-1}(\{M\in~U(2n)~:~MC_{2n} M^{t}C_{2n}^t=\lambda I_{2n},~det(M)=\lambda^n~~\&~|\lambda |=1\})
\end{align}

 I have there function algebras as follows:
\begin{align}
C(\widetilde{USp}(2n))&=\text{ continuous functions on }\Xi_{2n}(\{M\in~U(2n)~:~MK_{2n} M^{t}K_{2n}^t=\lambda I_{2n}~\&~|\lambda |=1\})\\
C(\widetilde{O}(n,\mathbb{R}))&=\text{ continuous functions on }\Xi_n(\{M\in~U(n)~:~MC_{n} M^{t}C_{n}^t=\lambda I_{n}~\&~|\lambda |=1\})\\
C(\widetilde{SO}(2n,\mathbb{R}))&=\text{ continuous functions on }\Xi_{2n}(\{M\in~U(2n)~:~MC_{2n} M^{t}C_{2n}^t=\lambda I_{2n},~det(M)=\lambda^n~~\&~|\lambda |=1\})
\end{align}

Let $q$ be a positive real number. Therefore It is known that $q$-deformations of $C(USp(2n))$, $C(O(n,\mathbb{R}))$ and $C(SO(n,\mathbb{R}))$ are studied in \cite{bookklisch97}. So in this paper I introduce the $q$-deformations of $C(\widetilde{USp}(2n))$, $C(\widetilde{O}(n,\mathbb{R}))$ and $C(\widetilde{SO}(2n,\mathbb{R}))$
\section{$q$-deformations of $C(\widetilde{USp}(2n))$, $C(\widetilde{O}(n,\mathbb{R}))$ and $C(\widetilde{SO}(n,\mathbb{R}))$}
It is well known that for a linear space $H$ with a basis $\{e_1,e_2,\cdots,e_n\}$ and a linear map $R:H\otimes H\longrightarrow H\otimes H$:
\begin{align*}
    R&:e_i\otimes e_j\to \sum_{k,l=1}^{n}R_{kl,ij}e_k\otimes e_l,\\
    \hat{R}=\text{ flip }\circ R&:,e_i\otimes e_j\to \sum_{k,l=1}^{n}R_{lk,ij}e_k\otimes e_l,
\end{align*}

there exists bialgebra $A_R$ (Faddeev-Reshetikhin-Takhtajan construction \cite{bookklisch97}) as follows.
$A_R$ is a bialgebra with unit, counit $\epsilon$ and coproduct $\Delta$ generated by the generators $V_{i,j}$ satisfying the relations:
\begin{align}
    \sum_{k,l=1}^{n}R_{ji,kl} V_{k,r}V_{l,s}&=\sum_{k,l=1}^{n}R_{lk,rs} V_{k,r}V_{l,s}\label{eq:R-commutation-relation}\\
    \Delta(V_{i,j})&=V_{i,k}\otimes V_{k,j}\nonumber\\
    \epsilon(V_{i,j})&=\delta_{i,j}\nonumber
\end{align}

Let $V=((V_{i,j}))\in M_{n}(A_R)$. Then one can assume $R,\hat{R}, V_1=V\otimes I_{n}, V_2=I_{n}\otimes V\in M_{n}(A_R)\otimes M_n(A_R)$ and from Eq-\ref{eq:R-commutation-relation}, I have \begin{align}
    RV_1V_2=V_2V_1R,~~~&\text{   or equivalently  }~~& \hat{R}V_1V_2=V_2V_1\hat{R}\label{eq:R-matrix-relation}
\end{align} 

First, I describe the situation for $SU_q(n)$ and $U_q(n)$. Then, I try to describe the corresponding situation for $USp_q(2n)$, $O_q(2n)$, $O_q(2n+1)$, and $SO_q(2n)$.
\subsection{Case 1: $SU_q(n)$ and $U_q(n)$}
Here I have \begin{align*}
    R_{ji,kl} = q^{\delta_{i,j}}\delta_{i,l}\delta_{j,k}+(q-\frac{1}{q})\delta_{i,k}\delta_{j,l}H(j-i)
\end{align*}
where $H$ is Heaviside symbol, that is $H(r)=1$ if $r>0$ and $H(r)=0$ if $r\leq 0$\label{Heaviside-symbol}.

Therefore I have $R_{ii,ii}=q$ for all $i$, $R_{ij,ij}=1$ for all $i\neq j$, $R_{ij,ji}=(q-\frac{1}{q})$ for all $i> j$ and $R_{ij,k,l}=0$ for other cases.

Using the left and right coactions of exterior algebra on $A_R$, I have the quantum determinant( see \cite{bookklisch97,Fiore94}) $\mathscr{D}_q\in A_R$ such that
\begin{align*}
    \mathscr{D}_q&=\sum_{\sigma\in\mathscr{S}_n} (-q)^{\ell(\sigma)}V_{{\sigma(1)},1}V_{{\sigma(2)},2}\cdots V_{{\sigma(n)},n}\nonumber\\&=\sum_{\sigma\in\mathscr{S}_n} (-q)^{\ell(\sigma)}V_{1,{\sigma(1)}}V_{2,{\sigma(2)}}\cdots V_{n,{\sigma(n)}}\nonumber\\
\end{align*}
and $\mathscr{D}_q$ commute with all $V_{i,j}$, $\Delta(\mathscr{D}_q)=\mathscr{D}_q\otimes \mathscr{D}_q$ and $\epsilon(\mathscr{D}_q)=1$. Now I can extend the bialgebra structure on $A_R[t]$ by defining $\Delta(t)=t\otimes t$ and $\epsilon(t)=1$. Then $\langle  \mathscr{D} - 1 \rangle$ and $\langle  t\mathscr{D} - 1\rangle$ are biideal of $A_R$ and $A_R[t]$ respectively. Therefore I have the bialgebra
$\mathbb{C}[SU_q(n)]=A_R/\langle  \mathscr{D} - 1 \rangle$ and  $\mathbb{C}[U_q(n)]=A_R[t]/\langle  t\mathscr{D} - 1\rangle$.
Consider the elements $s_{i,j}\in A_R$ such that
\begin{align*}
    s_{i,j}&= \sum_{\begin{matrix}\sigma\in\mathscr{S}_n\\\sigma(i)=j\end{matrix}} (-q)^{\ell_i(\sigma)} ~V_{\sigma(1),1}\cdots V_{\sigma(i-1),i-1}V_{\sigma(i+1),i+1}\cdots V_{\sigma(n),n}\\~\\
    &=\sum_{\begin{matrix}\sigma\in\mathscr{S}_n\\\sigma(j)=i\end{matrix}} (-q)^{\ell_j(\sigma)} ~V_{1,\sigma(1)}\cdots V_{j-1,\sigma(j-1)}V_{j+1,\sigma(j+1)}\cdots V_{n,\sigma(n)}\\
\end{align*}
where $\ell_i(\sigma)$ is the number of inversion of the bijection $\sigma|_{\{1,2,\cdots,i-1,i+1,\cdots,n\}}$.

Then I have $((V_{i,j}))\cdot ((s_{i,j}))=\mathscr{D}_qI_n$.  Using this relation and equation \eqref{eq:R-commutation-relation}, I have the antipode $S$ such that
\begin{align}
    \left.\begin{matrix}S(V_{i,j})&=&s_{i,j}\\S^2(V_{i,j})&=&q^{2(i-j)}V_{i,j}\end{matrix}\right\}&\text{ on } \mathbb{C}[SU_q(n)],\\~\\
    \left.\begin{matrix}S(V_{i,j})&=&ts_{i,j}\\S(t)&=&\mathscr{D}_q\\S^2(V_{i,j})&=&q^{2(i-j)}V_{i,j}\end{matrix}\right\}&\text{ on } \mathbb{C}[U_q(n)].
\end{align}
Then I can make them a Hopf-$*$-algebra by defining the involution such that $V_{i,j}^*=S(V_{j,i})$ and $~t^*=\mathscr{D}_q$.

The unitary corepresentation $((V_{i,j}))$ of $\mathbb{C}[SU_q(n)]$ and $((V_{i,j}))\oplus (\mathscr{D}_q^{-1})$ of $\mathbb{C}[U_q(n)]$ make them compact matrix quantum group algebra and their $C^*$-completion with bounded extension of $\Delta$ and $\epsilon$ make them compact quantum group.  For more details about $U_q(n)$, see \cite{koelink91}.

\subsection{Case 2: $\widetilde{USp}_q(2n)$ and ${USp}_q(2n)$}
Here consider $K_{q}=cdiag(q^{-\rho_1},q^{-\rho_2},\cdots,q^{-\rho_{n1}},-q^{-\rho_{n+1}},\cdots,-q^{\rho_{2n}})$ where $\rho_j=\left \{\begin{matrix}n+1-j&\text{ if }&j\leq n\\n-j&\text{ if }&n+1\leq j\leq 2n\end{matrix}\right.$ Therefore $K_{q}=cdiag(q^{-n},q^{-(n-1)},\cdots,q^{-1},-q^{1},\cdots,-q^{n})$. Also I have \begin{align*}
    R_{ij,mr} &= q^{\delta_{i,j}-\delta_{i,2n+1-j}}\delta_{i,m}\delta_{j,r}+\Big(q-\frac{1}{q}\Big)H(i-m)\Big[ \delta_{j,m}\delta_{i,r}+\Big(-1\Big)^{(\lfloor \frac{j-1}{n}\rfloor+\lfloor \frac{m-1}{n}\rfloor)}\delta_{j,2n+1-i}\delta_{m,2n+1-r}~~q^{-\rho_j+\rho_r}\Big]
\end{align*}
where $H$ is Heaviside symbol, that is $H(r)=1$ if $r>0$ and $H(r)=0$ if $r\leq 0$\ref{Heaviside-symbol}.

Using the left and right coactions of exterior algebra on $A_R$, I have the quantum determinant (see \cite{bookklisch97,Fiore94}) $\mathscr{D}_q\in A_R$ such that
\begin{align*}
 \mathscr{D}_{q}&=\sum_{\sigma\in\mathscr{S}_{2n}}(-q)^{\ell(\sigma)}q^{+r(\sigma)}
 t_{\sigma(1),1}t_{\sigma(2),2}\dots t_{\sigma(2n),2n}\\
    &=\sum_{\sigma\in\mathscr{S}_{2n}}(-q)^{\ell(\sigma)}q^{+r(\sigma)}
 t_{1,\sigma(1)}t_{2,\sigma(2)}\dots t_{2n,\sigma(2n)}
 \end{align*}
 where $\ell(\sigma)$ denotes the length of the permutation $\sigma$, i.e., the number of inversions of the permutation $\sigma$ and $r(\sigma)$ is the number of $i$ such that $j=\sigma^{-1}[2n+1-\sigma(i)]$ and $\sigma(j)>\sigma(i)$ for $1\leq i\leq n$.
Here $\mathscr{D}_q$ commute with all $V_{i,j}$, $\Delta(\mathscr{D}_q)=\mathscr{D}_q\otimes \mathscr{D}_q$ and $\epsilon(\mathscr{D}_q)=1$. But in this case I don't need quantum determinant.

Let $\mathcal{K}=\sum_{i,j=1}^{2n}\Big(-1\Big)^{(\lfloor \frac{i-1}{n}\rfloor+\lfloor \frac{j-1}{n}\rfloor)}q^{\rho_i-\rho_j}E_{2n+1-i,j}\otimes E_{i,2n+1-j}\in M_{2n}(A_R)\otimes M_{2n}(A_R)$. Therefore I have $\mathcal{K}=I_{2n}-(q-\frac{1}{q})^{-1}(\hat{R}-\hat{R}^{-1})$. So  \begin{align*}
    \mathcal{K}=\frac{1-(q-\frac{1}{q})^{-1}(q^{2n+1}-q^{-2n-1})}{(-q-q^{-2n-1})(q^{-1}-q^{-2n-1})}[\hat{R}^2-(q-\frac{1}{q})\hat{R}-I],
\end{align*}
polynomial of $\hat{R}$. From Eq-\ref{eq:R-matrix-relation}, I have 
\begin{align}
    \mathcal{K}V_1V_2=V_2V_1\mathcal{K}.\label{eq:K-commutation-relation-USp}
\end{align} 

Consider elements $s_{i,j}\in A_R$ such that
\begin{align*}
    s_{i,j}&= \Big(-1\Big)^{(\lfloor \frac{i-1}{n}\rfloor +\lfloor \frac{j-1}{n}\rfloor)}q^{\rho_j-\rho_i} V_{2n+1-i,2n+1-j}
\end{align*}

Here I have an element $\mathscr{Q}_q\in A_R$ such that
\begin{align*}
 \mathscr{Q}_{q}=\sum_{k=1}^{2n}V_{i,k}s_{k,i}=\sum_{k=1}^{2n}s_{i,k}V_{k,i}&\text{ for all }~i
 \end{align*}
 and $\mathscr{Q}_q$ commute with all $V_{i,j}$, $\Delta(\mathscr{Q}_q)=\mathscr{Q}_q\otimes \mathscr{Q}_q$ and $\epsilon(\mathscr{Q}_q)=1$. For details see \cite{bookklisch97}.

Now one can extend the bialgebra structure on $A_R[t]$ by defining $\Delta(t)=t\otimes t$ and $\epsilon(t)=1$. Then 
\begin{align*}
    I_1&=\langle  \sum_{k=1}^{2n} V_{i,k}s_{k,j} - \delta_{i,j},\sum_{k=1}^{2n} s_{i,k}V_{k,j} - \delta_{i,j} ~:i,j\in\{1,2,\cdots, 2n\}\rangle&\text{  and  }\\
    J_1&=\langle   t(\sum_{k=1}^{2n} V_{i,k}s_{k,j}) - \delta_{i,j},t(\sum_{k=1}^{2n} s_{i,k}V_{k,j}) - \delta_{i,j} ~:i,j\in\{1,2,\cdots, 2n\} \rangle
\end{align*} are biideal of $A_R$ and $A_R[t]$ respectively. Therefore I have the bialgebra
$\mathbb{C}[USp_q(2n)]=A_R/I_1$ and  $\mathbb{C}[\widetilde{USp}_q(2n)]=A_R[t]/J_1$.

Then I have $((V_{i,j}))\cdot ((s_{i,j}))=\mathscr{Q}_qI_n$.  Using this relation and equation \eqref{eq:K-commutation-relation-USp}, I have the antipode $S$ such that
\begin{align}
    \left.\begin{matrix}S(V_{i,j})&=&s_{i,j}\\S^2(V_{i,j})&=&q^{2(\rho_j-\rho_i)}V_{i,j}\end{matrix}\right\}&\text{ on } \mathbb{C}[USp_q(2n)],\\~\\
    \left.\begin{matrix}S(V_{i,j})&=&ts_{i,j}\\S(t)&=&\mathscr{Q}_q\\S^2(V_{i,j})&=&q^{2(\rho_j-\rho_i)}V_{i,j}\end{matrix}\right\}&\text{ on } \mathbb{C}[\widetilde{USp}_q(2n)].
\end{align}
Then I can make them a Hopf-$*$-algebra by defining the involution such that $V_{i,j}^*=S(V_{j,i})$ and $~t^*=\mathscr{Q}_q$.

The unitary corepresentation $((V_{i,j}))$ of $\mathbb{C}[USp_q(2n)]$ and $((V_{i,j}))\oplus (\mathscr{Q}_q^{-1})$ of $\mathbb{C}[\widetilde{USp}_q(2n)]$ make them compact matrix quantum group algebra and their $C^*$-completion with bounded extension of $\Delta$ and $\epsilon$ make them compact quantum group.

Here one has a $C^*$ morphism ${\varphi}_{q,n}:C({USp}_q(2n+2))\to C({USp}_q(2n))$ defined by
\begin{align*}
	{\varphi}_{q,n}(V_{i,j})=\left\{\begin{matrix}
		V_{i-1,j-1}&\text{ if }&2\leq i,j\leq 2n+1\\ \delta_{i,j}&&\text{ otherwise }
	\end{matrix}\right.
\end{align*}

Moreover, there exist $C^*$-morphisms $\varrho_{q,n}:C(\widetilde{USp}_q(2n))\to C({USp}_q(2n))$ and $\widetilde{\varphi}_{q,n}:C(\widetilde{USp}_q(2n+2))\to C(\widetilde{USp}_q(2n))$ satisfying
\begin{align*}
	\varrho_{q,n}(V_{i,j})&=V_{i,j},\\ \varrho_{q,n}(\mathscr{Q}_q)&=1,\\
	\widetilde{\varphi}_{q,n}(V_{i,j})&=\left\{\begin{matrix}
		V_{i-1,j-1}&\text{ if }&2\leq i,j\leq 2n+1\\ \delta_{i,j}&&\text{ otherwise }
	\end{matrix}\right.,\\ \widetilde{\varphi}_{q,n}(\mathscr{Q}_q)&=\mathscr{Q}_q.
\end{align*}

\subsection{Case 3: ${O}_q(N)$, ${SO}_q(N)$, $\widetilde{O}_q(N)$ and $\widetilde{SO}_q(2n)$}
For the even case $N=2n$, consider $C_{q}=cdiag(q^{-\rho_1},q^{-\rho_2},\cdots,q^{-\rho_{n}},q^{-\rho_{n+1}},\cdots,q^{-\rho_{2n}})$ where\newline $\rho_j=\left \{\begin{matrix}n-j&\text{ if }&j\leq n\\n+1-j&\text{ if }&n+1\leq j\leq 2n\end{matrix}\right.$ Therefore $C_{q}=cdiag(q^{-(n-1)},q^{-(n-2)},\cdots,1,1,\cdots,q^{n-1})$.

For the odd case $N=2n+1$, consider $C_{q}=cdiag(q^{-\rho_1},q^{-\rho_2},\cdots,q^{-\rho_{n}},q^{-\rho_{n+1}},q^{-\rho_{n+2}},\cdots,q^{-\rho_{2n+1}})$ where \newline $\rho_j=\left \{\begin{matrix}n+\frac{1}{2}-j&\text{ if }&1\leq j\leq n\\0&\text{ if }&j=n+1\\2n+\frac{3}{2}-j&\text{ if }&n+2\leq j\leq 2n+1\end{matrix}\right.$  Therefore $C_{q}=cdiag(q^{-(n-\frac{1}{2})},q^{-(n-\frac{3}{2})},\cdots,q^{-\frac{1}{2}},1,q^{\frac{1}{2}}\cdots,q^{n-\frac{1}{2}})$.

Here I have \begin{align*}
    R_{ij,mr} &= q^{\delta_{i,j}-\delta_{i,N+1-j}}\delta_{i,m}\delta_{j,r}+\Big(q-\frac{1}{q}\Big)H(i-m)\Big[ \delta_{j,m}\delta_{i,r}-\delta_{j,N+1-i}\delta_{m,N+1-r}~~q^{-\rho_j+\rho_r}\Big]
\end{align*}
where $H$ is Heaviside symbol, that is $H(r)=1$ if $r>0$ and $H(r)=0$ if $r\leq 0$\ref{Heaviside-symbol}.

Using the left and right coactions of exterior algebra on $A_R$, one has the quantum determinant $\mathscr{D}_q\in A_R$ given below.

For $N=2n$ I have the expression of the quantum determinant (See \cite{bookklisch97}, \cite{Fiore94})
\begin{align*}
 \mathscr{D}_{q}&=\sum_{\sigma\in\mathscr{S}_{2n}}(-q)^{\ell(\sigma)}q^{-r(\sigma)}
 t_{\sigma(1),1}t_{\sigma(2),2}\dots t_{\sigma(2n),2n}\\
    &=\sum_{\sigma\in\mathscr{S}_{2n}}(-q)^{\ell(\sigma)}q^{-r(\sigma)}
 t_{1,\sigma(1)}t_{2,\sigma(2)}\dots t_{2n,\sigma(2n)}
 \end{align*}
 where $\ell(\sigma)$ denotes the length of the permutation $\sigma$, i.e., the number of inversions of the permutation $\sigma$ and $r(\sigma)$ is the number of $i$ such that $j=\sigma^{-1}[2n+1-\sigma(i)]$ and $\sigma(j)>\sigma(i)$ for $1\leq i\leq n$.

 For $N=2n+1$, I first note meaning of some symbols.

Take  $\phi\neq Y=\{j_1<j_2<\cdots<j_k\}\subsetneqq\{1,2,\cdots,n\}$. Let $\{1,2,\cdots,n\}\setminus Y=\{j_{k+1}<j_{k+2}<\cdots<j_{n}\}$ and 
 $\sigma$ be rearrangement of $Y\cup \{2n+2-j_k<2n+2-j_{k-1}<\cdots<2n+2-j_1\}\cup \{n+1\}$ where $n+1$ occure $2n+1-2k$ times and oher elements occure once. Suppose $\sigma(i)$ be the $i-th$ positional value of the rearrangement. $\mathscr{S}_Y$ be collection of all such rearrangements. Then $\ell(\sigma)$ denotes the number of inversion of the rearrangement $\sigma$ and $r(\sigma)$ is the number of $i$ such that $1\leq i\leq n$ and $\sigma(i)\leq n$. $\ell_1(Y)$ is number of inversion of permutation $\begin{pmatrix}1 & 2 & 3 & \cdots & n-1 & n \\j_1 & j_2 & j_3 & \cdots &  j_{n-1}  & j_n\end{pmatrix}$ and $\ell_2(Y)$ is number of inversion of permutation $\begin{pmatrix}n+2 & n+3 & n+4 & \cdots & 2n & 2n+1 \\2n+2-j_n & 2n+2-j_{n-1} & 2n+2-j_{n-2} & \cdots &  2n+2-j_{2}  & 2n+2-j_1\end{pmatrix}$
 Then the expression for the quantum determinant is given by (see \cite{bookklisch97,Fiore94}):
 \begin{align*}\mathscr{D}_{q}&=\sum_{\sigma\in\mathscr{S}_{2n+1}}(-q)^{\ell(\sigma)}q^{-r(\sigma)}
 V_{1,\sigma(1)}V_{2,\sigma(2)}\dots V_{2n+1,\sigma(2n+1)}\\&+\sum_{\sigma\in\mathscr{S}_Y}(-q)^{\ell(\sigma)}q^{-r(\sigma)}(-q)^{\ell_1(Y)}(-q)^{\ell_2(Y)}(\sqrt{q}-\frac{1}{\sqrt{q}})^{n-k}q^{j_{k+1}+j_{k+2}+\cdots+j_{n}-n(n-k)}(-q)^{(n-k)(n-k)}(n-k)!
 \\&\cdot V_{1,\sigma(1)}V_{2,\sigma(2)}\dots V_{2n+1,\sigma(2n+1)}\end{align*}

In both cases, $\mathscr{D}_q$ commute with all $V_{i,j}$, $\Delta(\mathscr{D}_q)=\mathscr{D}_q\otimes \mathscr{D}_q$ and $\epsilon(\mathscr{D}_q)=1$. But I are not interested on $\mathscr{D}_q$ or odd case.

Let $\mathcal{K}=\sum_{i,j=1}^{N}q^{\rho_i-\rho_j}E_{N+1-i,j}\otimes E_{i,N+1-j}\in M_{N}(A_R)\otimes M_{N}(A_R)$. Therefore I have $\mathcal{K}=I_{2n}-(q-\frac{1}{q})^{-1}(\hat{R}-\hat{R}^{-1})$. So  \begin{align*}
    \mathcal{K}=\frac{1+(q-\frac{1}{q})^{-1}(q^{N-1}-q^{1-N})}{(q^{1-N}-q)(q^{-1}+q^{1-N})}[\hat{R}^2-(q-\frac{1}{q})\hat{R}-I],
\end{align*}
polynomial of $\hat{R}$. From Eq-\ref{eq:R-matrix-relation}, I have 
\begin{align}
    \mathcal{K}V_1V_2=V_2V_1\mathcal{K}.\label{eq:K-commutation-relation}
\end{align}

Consider elements $s_{i,j}\in A_R$ such that
\begin{align*}
    s_{i,j}&= q^{\rho_j-\rho_i} V_{N+1-i,N+1-j}
\end{align*}

Here I have an element $\mathscr{Q}_q\in A_R$ such that
\begin{align*}
 \mathscr{Q}_{q}=\sum_{k=1}^{N}V_{i,k}s_{k,i}=\sum_{k=1}^{N}s_{i,k}V_{k,i}&\text{ for all }~i
 \end{align*}
 and $\mathscr{Q}_q$ commute with all $V_{i,j}$, $\Delta(\mathscr{Q}_q)=\mathscr{Q}_q\otimes \mathscr{Q}_q$ and $\epsilon(\mathscr{Q}_q)=1$. For details see \cite{bookklisch97}.

Now one can extend the bialgebra structure on $A_R[t]$ by defining $\Delta(t)=t\otimes t$ and $\epsilon(t)=1$. Then 
\begin{align*}
    I_2&=\langle  (\sum_{k=1}^{N} V_{i,k}s_{k,j}) - \delta_{i,j},(\sum_{k=1}^{N} s_{i,k}V_{k,j}) - \delta_{i,j} ~:i,j\in\{1,2,\cdots, N\}\rangle,&\\
    I_3&=\langle  \mathscr{D}_q-1, (\sum_{k=1}^{N} V_{i,k}s_{k,j}) - \delta_{i,j},(\sum_{k=1}^{N} s_{i,k}V_{k,j}) - \delta_{i,j} ~:i,j\in\{1,2,\cdots, N\}\rangle,&\\
    J_2&=\langle   t(\sum_{k=1}^{N} V_{i,k}s_{k,j}) - \delta_{i,j},t(\sum_{k=1}^{N} s_{i,k}V_{k,j}) - \delta_{i,j} ~:i,j\in\{1,2,\cdots, N\} \rangle,&\\
    J_3&=\langle  t^n\mathscr{D}_q-1, t(\sum_{k=1}^{2n} V_{i,k}s_{k,j}) - \delta_{i,j},t(\sum_{k=1}^{2n} s_{i,k}V_{k,j}) - \delta_{i,j} ~:i,j\in\{1,2,\cdots, 2n\} \rangle
\end{align*} are biideal of $A_R$ and $A_R[t]$ respectively. Therefore I have the bialgebra
$\mathbb{C}[O_q(N)]=A_R/I_2$, $\mathbb{C}[SO_q(N)]=A_R/I_3$,  $\mathbb{C}[\widetilde{O}_q(N)]=A_R[t]/J_2$, and $\mathbb{C}[\widetilde{SO}_q(2n)]=A_R[t]/J_3$

Then I have $((V_{i,j}))\cdot ((s_{i,j}))=\mathscr{Q}_qI_n$.  Using this relation and equation \eqref{eq:K-commutation-relation}, I have the antipode $S$ such that
\begin{align}
    \left.\begin{matrix}S(V_{i,j})&=&s_{i,j}\\S^2(V_{i,j})&=&q^{2(\rho_j-\rho_i)}V_{i,j}\end{matrix}\right\}&\text{ on } \mathbb{C}[O_q(N)],\\~\\
    \left.\begin{matrix}S(V_{i,j})&=&ts_{i,j}\\S(t)&=&\mathscr{Q}_q\\S^2(V_{i,j})&=&q^{2(\rho_j-\rho_i)}V_{i,j}\end{matrix}\right\}&\text{ on } \mathbb{C}[\widetilde{O}_q(N)].
\end{align}

As two sided ideal $\langle\mathscr{D}_q-1\rangle$ is a Hopf ideal of $\mathbb{C}[O_q(N)]$, I have $\mathbb{C}[SO_q(N)]=\mathbb{C}[O_q(N)]/\langle\mathscr{D}_q-1\rangle$ is Hopf algebra. Similarly $\mathbb{C}[\widetilde{SO}_q(2n)]=\mathbb{C}[\widetilde{O}_q(2n)]/\langle  t^n\mathscr{D}_q-1\rangle$ is also a Hopf algebra.

Then I can make them a Hopf-$*$-algebra by defining the involution such that $V_{i,j}^*=S(V_{j,i})$and $~t^*=\mathscr{Q}_q$ (For $\mathbb{C}[O_q(N)]$ I have $\mathscr{D}_q^*=\mathscr{D}_q$).

The unitary corepresentation $((V_{i,j}))$ of $\mathbb{C}[O_q(N)]$ and $\mathbb{C}[SO_q(N)]$ and $((V_{i,j}))\oplus (\mathscr{Q}_q^{-1})$ of $\mathbb{C}[\widetilde{O}_q(N)]$ and $\mathbb{C}[\widetilde{SO}_q(2n)]$ make them compact matrix quantum group algebra and their $C^*$-completion with bounded extension of $\Delta$ and $\epsilon$ make them compact quantum group.

Similarly one has $C^*$ morphisms ${\varphi}_{q,n}:C({O}_q(N+2))\to C({O}_q(N))$ defined by
\begin{align*}
	{\varphi}_{q,n}(V_{i,j})=\left\{\begin{matrix}
		V_{i-1,j-1}&\text{ if }&2\leq i,j\leq N+1\\ \delta_{i,j}&&\text{ otherwise }
	\end{matrix}\right.
\end{align*} where $N=2n$ or $2n+1$

Moreover, there are corresponding morphisms $\varrho_{q,n}:C(\widetilde{O}_q(N))\to C({O}_q(N))$ and $\widetilde{\varphi}_{q,n}:C(\widetilde{O}_q(N+2))\to C(\widetilde{O}_q(N))$ satisfying
\begin{align*}
	\varrho_{q,n}(V_{i,j})&=V_{i,j},\\ \varrho_{q,n}(\mathscr{Q}_q)&=1,\\
	\widetilde{\varphi}_{q,n}(V_{i,j})&=\left\{\begin{matrix}
		V_{i-1,j-1}&\text{ if }&2\leq i,j\leq N+1\\ \delta_{i,j}&&\text{ otherwise }
	\end{matrix}\right.,\\ \widetilde{\varphi}_{q,n}(\mathscr{Q}_q)&=\mathscr{Q}_q
\end{align*} assuming $N=2n$ or $2n+1$.

A similar construction also holds for the $SO_q(2n)$ case.

\section{All $C^*$ irreducible representations}
Consider the set $T=\{t_{i,j}~:~1\leq i,j\leq n\}$. Let $Pol_k(T)$ be the non-commutative homogeneous polynomial of degree $k$. Consider the following sets of relations:
\begin{align*}
    R_1&=\{E_{1,i}~:~E_{1,i}\in Pol_{r_i}(T) \text{ and }1\leq i\leq p\}\\
    R_0&=\{s_{i,j}\in Pol_{k}(T)~:~1\leq i,j\leq n, \text{ for some fixed }k~\} \\
    R_3&=\{E_{3,i}-1~:~E_{3,i}\in Pol_{m_i(k+1)}(T) \text{ and }1\leq i\leq  c\}\\
    R_4&=\{\sum_{k=1}^nt_{i,k}s_{j,k}-\delta_{i,j},\sum_{k=1}^ns_{k,i}t_{k,j}-\delta_{i,j}~:~ s_{i,j}\in R_0\}\\
    S_0&= \{\mho^{-1} t_{i,j}-t_{i,j}\mho^{-1}\}\\
    S_3&=\{E_{3,i}(\mho^{-1})^{m_i}-1~:~ E_{3,i}-1\in R_3\}\\
    S_4&=\{(\sum_{k=1}^nt_{i,k}s_{j,k})\mho^{-1}-\delta_{i,j},(\sum_{k=1}^ns_{k,i}t_{k,j})\mho^{-1}-\delta_{i,j}~:~ s_{i,j}\in R_0\}\\
\end{align*}

\begin{theorem}
Let \( A \) be the universal unital \( C^* \)-algebra generated by a set of elements \( T \), subject to the relations \( R_1 \sqcup R_3 \sqcup R_4 \sqcup \{ t_{i,j}^* - s_{i,j}~:~s_{i,j}\in R_0\}\). Then there exists a universal unital \( C^* \)-algebra \( Z \), generated by the set \( T \sqcup \{ \mho^{-1} \} \), satisfying the relations \( S_0 \sqcup R_1  \sqcup S_3 \sqcup S_4\sqcup \{ t_{i,j}^* - \mho^{-1}s_{i,j}~:~s_{i,j}\in R_0\} \).
Note that any one of $R_1$ and $R_3$ may be empty.
Moreover, every irreducible representation of \( Z \) is of the form \( \pi_{\alpha,\lambda} \), for some irreducible representation \( \pi_\alpha \) of \( A \), where
\[
\pi_{\alpha,\lambda}(t_{i,j}) = \lambda \pi_\alpha(t_{i,j}) \quad \text{and} \quad \pi_{\alpha,\lambda}(\mho^{-1}) = \overline{\lambda}^{k+1}.
\]
\end{theorem}

\begin{proof}Consider any representation $\pi_{\alpha}$ of $A$. Let $\lambda$ be a complex number with \( |\lambda| = 1 \).Define a representation \( \psi: Z \to \mathcal{B}(H) \) by
\[
\psi(t_{i,j}) := \lambda \pi_\alpha(t_{i,j}), \quad \psi(\mho^{-1}) := \overline{\lambda}^{k+1},
\]
where $k$ be a positive integer such that $s_{i,j}\in Pol_k(T)$. Therefore I have $\psi(s_{i,j}) = \lambda^k \pi_\alpha(s_{i,j})= \lambda^k [\pi_\alpha(t_{i,j})]^*$
Then $\psi$ is the admissible representation of the relations \( S_0 \sqcup R_1  \sqcup S_3 \sqcup S_4\sqcup \{ t_{i,j}^* - \mho^{-1}s_{i,j}~:~s_{i,j}\in R_0\} \) such that $\psi(t_{i,j}^*)=\psi(\mho^{-1}s_{i,j})$.

Consider the set $\Gamma=\{\pi:\pi\text{ satisfies } S_0 \sqcup R_1  \sqcup S_3 \sqcup S_4\sqcup \{ t_{i,j}^* - \mho^{-1}s_{i,j}~:~s_{i,j}\in R_0\} \}$.

Define the norm $\Arrowvert p(t_{i,j},\mho^{-1})\Arrowvert_u=\underset{\pi\in\Gamma}{Sup}\Arrowvert p(t_{i,j},\mho^{-1})\Arrowvert$. From $S_4\sqcup \{ t_{i,j}^* - \mho^{-1}s_{i,j}~:~s_{i,j}\in R_0\}$, I have $\sum_kt_{i,k}t_{j,k}^*=\delta_{i,j}$. Then I have $\Arrowvert t_{i,j}\Arrowvert_u\leq 1$ for all $i,j$ and $\Arrowvert \mho^{-1}\Arrowvert_u= 1$. Therefore, the universal $C^*$ algebra exist.

Consider $\zeta$ be any irreducible representation of $Z$. Therefore the unitary $\zeta(\mho^{-1})$ commutes with every generators. Therefore  $\zeta(\mho^{-1})=\mu I$ for some unit norm complex number $\mu$. Let $\lambda$ be a complex number such that $\lambda^{k+1}=\mu$. Define $\wp(t_{i,j})=\lambda \zeta(t_{i,j})$. Then $\wp(s_{i,j})=\lambda^k\zeta(s_{i,j})=\lambda^k\zeta(\mho)\zeta(t_{i,j})^*=\overline{\lambda}\zeta(t_{i,j})^*$. So $\wp$ is also a representation of $A$. Since multiplication by a scalar does not change invariant subspaces, $\wp$ is irreducible.

Therefore, every irreducible representation of \( Z \) is of the prescribed form.
\end{proof}

\begin{remark}
  Therefore, I obtain all irreducible $C^*$-representations of $\mathbb{C}[\widetilde{USp}_q(2n)]$, $\mathbb{C}[\widetilde{O}_q(N)]$, and $\mathbb{C}[\widetilde{SO}_q(2n)]$ from the irreducible representations of $\mathbb{C}[USp_q(2n)]$, $\mathbb{C}[O_q(N)]$, and $\mathbb{C}[SO_q(2n)]$, respectively. For more details on the irreducible representations of $C(SU_q(n))$ and $C(USp_q(2n))$, see \cite{korsoi98}.
\end{remark}

\section*{Acknowledgment}
The author was supported by the Visiting Scientist position at Indian Statistical Institute from July 2025 to October 2025 (No. DO/2025/255) while working on this paper.


\end{document}